\documentclass[graybox]{svmult}

\usepackage{mathptmx}       
\usepackage{helvet}         
\usepackage{courier}        
\usepackage{type1cm}        
\usepackage{makeidx}         
\usepackage{graphicx}        
\usepackage{multicol}        
\usepackage[bottom]{footmisc}

\usepackage{amsmath}
\usepackage{amssymb}

\makeindex             

\newcommand{\HH}{\mathcal{H}}

\def\V{\mathcal{V}}

\newcommand{\R}{\mathbb{R}}
\newcommand{\Z}{\mathbb{Z}}
\newcommand{\N}{\mathbb{N}}
\newcommand{\CC}{\mathbb{C}}

\newcommand{\la}{\langle}
\newcommand{\ra}{\rangle}
\newcommand{\G}{ G}
\newcommand{\wi}{\widetilde}
\newcommand{\Om}{\Omega}
\def\supp{{\rm supp\,}}

\begin{document}

\title*{Dynamical sampling and systems from  iterative actions of operators}
\author{Akram Aldroubi and Armenak Petrosyan}
\institute{Akram Aldroubi \at Vanderbilt University, Nashville, TN 37240, USA \email{akram.aldroubi@vanderbilt.edu}
\and Armenak Petrosyan \at Vanderbilt University, Nashville, TN 37240, USA \email{armenak.petrosyan@vanderbilt.edu}}
%
%
\maketitle

\abstract*{Each chapter should be preceded by an abstract (10--15 lines long) that summarizes the content. The abstract will appear \textit{online} at \url{www.SpringerLink.com} and be available with unrestricted access. This allows unregistered users to read the abstract as a teaser for the complete chapter. As a general rule the abstracts will not appear in the printed version of your book unless it is the style of your particular book or that of the series to which your book belongs.
Please use the 'starred' version of the new Springer \texttt{abstract} command for typesetting the text of the online abstracts (cf. source file of this chapter template \texttt{abstract}) and include them with the source files of your manuscript. Use the plain \texttt{abstract} command if the abstract is also to appear in the printed version of the book.}

\abstract{In this chapter, we review some of the recent developments and prove new results concerning frames and Bessel systems generated by iterations of the form $\{A^ng: g\in \G,\, n=0,1,2,\dots \}$, where $A$ is a bounded linear operator on a separable complex Hilbert space $ \HH $ and $\G$ is a countable set of vectors in $ \HH $. The system of iterations  mentioned above was motivated from the so called {\em dynamical sampling problem}. In dynamical sampling, an unknown function $f$ and its future states $A^nf$ are coarsely sampled at each time level $n$, $0\le n< L$, where $A$ is an evolution operator that drives the system. The goal  is to recover $f$ from these space-time samples. }

\section{Introduction}

The typical dynamical sampling problem is finding spatial positions $X=\{x_i \in \R^d: i\in I\}$ that allow the  reconstruction of an unknown function $f \in \HH \subset L^2\left(\R^d\right)$ from samples of the function at spatial positions $x_i\in X$ and subsequent samples of the functions $A^nf$, $n=0,\cdots,L$, where $A$ is an evolution operator and $n$ represents time. For example, $f$ can be the temperature at time $n=0$, $A$ the heat evolution operator, and $A^nf$ the temperature at time $n$. The problem is then to find spatial sampling   positions $X\subset \R^d$, and end time $L$, that allow the determination of the initial temperature $f$ from samples $\{f|_X, (Af)|_X,\cdots,(A^Lf)|_X\}$. For the heat evolution operator, the  problem has been considered by Vetterli and Lu \cite {LDV11,LV09} and inspired our  research in dynamical sampling, see e.g., \cite {ACCMP15,ACMT14,ADK13,ADK15}.

\subsection{The dynamical sampling problem}
\label {DSP}
		
		Let $\HH$ be a separable (complex) Hilbert space,
		$f\in \HH$ be an unknown vector 
		and  $f_n\in \HH$ be the state of the system at time $n$.		 We assume  
			$$f_0=f,\;\;f_n=Af_{n-1}=A^nf$$
		where $A$ is a known bounded operator on $\HH$.
		Given the measurements
			\begin{equation} \label{samples}
			\la A^nf,g \ra \;\text{for}\; \;0\leq n<L(g), \; \;g\in \G
			\end{equation}		
		where $\G$ is a countable subset of $\HH$ and $L:\G\mapsto \N\cup \{\infty\}$ is a function, the dynamical sampling problem is to recover the vector $f\in \HH$ from the measurements \eqref{samples}. 	It is important that the recovery of $f$ be robust to noise. Thus, we also require that  the sampling data allow the recovery of $f$ in a stable way. Equivalently, for any $f\in \HH$, the samples must satisfy the stability condition
			$$C_1\|f\|^2\leq\sum_{g \in \G}\sum_{n=0}^{L(g)}|<A^nf,g>|^2\le C_2\|f\|^2,$$ 
for some $C_1,C_2>0$ absolute constants.

A related problem for band-limited signals in $\R^2$   (i.e., the Paley Wiener spaces $PW_\sigma$) with time varying sampling locations corresponding to trajectories  but time-independent function  can be found in  \cite {GRUV15}.

\subsection{Dynamical sampling for diagonalizable operators in $l^2(\N)$}
When the Hilbert space is $\HH=l^2(\N)$, and when the operator $A$ is equivalent to a diagonal matrix $D$, i.e., $A^\ast=B^{-1}DB$ where $D=\sum_j\lambda_j P_j$ is an infinite diagonal matrix, then a characterization of the set of sampling $I \subset \N$ such that any $f \in \HH$ can be recovered from the data $Y=\{f(i), (Af)(i),\cdots,(A^{l_i}f)(i): \, i \in I \}$ is obtained in  \cite {ADK15}. 
 
The results are stated in terms of   vectors $b_i$ that are the columns of $B$ corresponding to the sampling positions $i \in I$,  the projections  $P_j$  that are diagonal infinite matrices whose non-zero diagonals are all ones and correspond to the  projection on the eigenspace of $D$ associated to $\lambda_j$, and the smallest integers $l_i$ such that the sets $\{b_i, Db_i,\dots,D^{l_i}b_i\}$ are minimal \cite {ACMT14}.

\begin {theorem}\label {TAD1}

Let $A^\ast=B^{-1}DB$, and  let $\{b_i: i\in I  \}$ be the column vectors of $B$ whose indices belong to $I$. Let $l_i$ be the smallest integers such that the set $\{b_i, Db_i,\dots,D^{l_i}b_i\}$ is minimal. Then  any vector $f \in l^2(\N)$ can be recovered from  the  samples 
$$Y=\{f(i), Af(i), \dots, A^{l_i}f(i): i \in I\}$$  if and only if for each $j$, $\{P_j(b_i):i \in I\}$  is complete in the range $E_j$ of $P_j$. 
\end{theorem}
Although Theorem \ref {TAD1} characterizes the sets $I\subset \N$ such that  recovery of any $f \in l^2(\N)$ is possible, it does not provide conditions for stable recovery, i.e., recovery that is robust to noise. Results on the stable recovery are obtained for the case when $I$ is finite \cite {ACMT14}. Stable recovery is also  obtained when $\HH=l^2(\Z)$, $A$ is a convolution operator, and $I$ is a union of  uniform grids  \cite {ADK15}.  For shift-invariant spaces, and union of uniform grids, stable recovery results can be found in  \cite {AADP13}. Obviously, recovery and stable recovery are equivalent in finite dimensional spaces \cite {ADK13}. In \cite{AP15} the case when the locations of the sampling positions are allowed to change is considered. 

\subsubsection {Connections with other fields and applications} 

The dynamical sampling problem has similarity with wavelets  \cite {BJ02,CM11,D92,HW96,M98,OS04,Pes13,SN96}. In dynamical sampling an  operator $A$ is
applied iteratively to the function $f$ producing the functions $f_n=A^nf$. $f_n$ is then, typically, sampled coarsely at each level $n$. Thus, $f$ cannot be recovered from samples at any  single time-level. But, similarly to the wavelet transform, the combined data at all time levels is required to reproduce $f$. 
However, unlike the wavelet transform,  there is a single operator $A$ instead of two complementary operators $L$ (the lowpass operator) and $H$ (the high pass operator). Moreover, $A$ is imposed by the constraints  of the problem, rather than designed, as in the case of $L$ and $H$ in wavelet theory. Finally, in dynamical sampling,  the spatial-sampling grids is not required to be regular.

In inverse problems, given an operator $B$ that represents a physical process,
the goal is to recover a function $f$ from the observation  $Bf$. Deconvolution or debluring  are prototypical examples. When
$B$ is not bounded below, the problem is considered  ill-posed  (see e.g., \cite {N11}).  The dynamical
sampling problem can be viewed as an inverse problem when the operator $B$ is the
result of applying the operators $S_{X_0}, S_{X_1}A, S_{X_2}A^2,\dots, S_{X_L}A^L$, where $S_{X_l}$ is the sampling operator at time $l$ on the set $X_l$, i.e., $B_X=[S_{X_0}, S_{X_1}A, S_{X_2}A^2,\dots, S_{X_L}A^L]^T$. However, unlike  the typical inverse problem, in dynamical sampling  the goal is to find conditions on $L$, $\{X_i : \, i=0,\dots,L\}$, and $A$, such that $B_X$ is injective, well conditioned, etc.

The dynamical sampling problem has connections and applications to other areas of mathematics including, Banach algebras, $C^*$-algebras, spectral theory of normal operators, and frame theory  \cite {ABK08,CKL08, Con96, FS10, G04,GL04,Pes15,S08ACM}. 
 
Dynamical sampling has potential applications in plenacoustic sampling, on-chip sensing,  data center temperature sensing, neuron-imaging, and satellite  remote sensing, and more generally to   Wireless Sensor Networks (WSN).   In wireless sensor networks, measurement  devices are distributed to gather information about a physical quantity 
to be monitored, such as temperature, pressure, or pollution  \cite{HRLV10,LDV11,RCLV11,LV09,RMG12}. The  goal  
 is to exploit the evolutionary structure and the placement of sensors to reconstruct an unknown  field. When it is not possible to place sampling devices at the desired locations (e.g., when there are not enough devices), then the desired information field can be recovered by placing the sensors elsewhere and taking advantage of the evolution process to recover the signals at the relevant locations. Even when the placement of sensors is not constrained,  if the cost of a
sensor is expensive relative to the cost of activating the sensor, then
the relevant information may be recovered with fewer sensors placed judiciously and activated  frequently. Super resolution is another applications when a  evolutionary process acts on the signal of interest.

\subsection{Contribution} In this chapter, we further develop the case of iterative systems generated by the iterative actions of normal operators which was studied in \cite {ACCMP15,ACMT14}. This is done in Section \ref {ANOII}. In Section  \ref  {NRGBO} we study the case of general  iterative systems generated by the iterative actions of  operators that are not necessarily normal.

		\begin{figure}[t]
			\sidecaption
			\includegraphics[scale=.45]{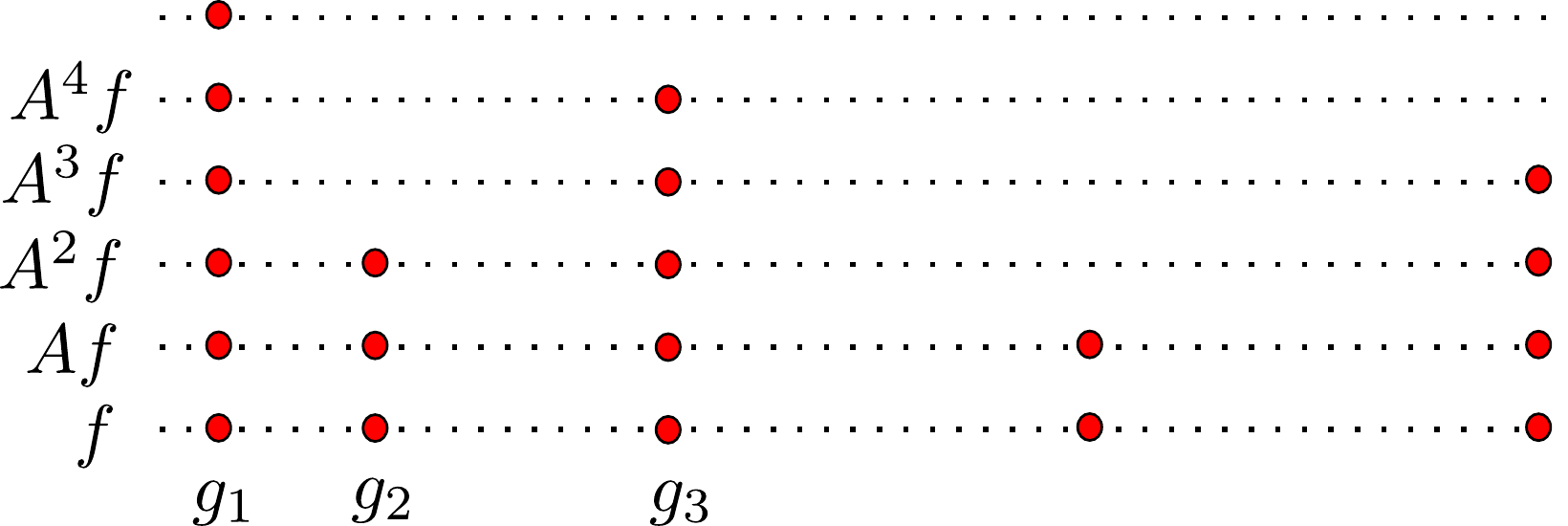}
			%
			%
			\caption{The dynamical sampling procedure when the sampling happens at different time levels.}
			\label{fig:1}       
		\end{figure}

\section{Frame and Bessel properties of systems from iterative actions of operators}
In this section we review some results from \cite {ACCMP15, ACMT14} on the iterative actions of normal operators, prove some new results for this case and generalize  several results to the case where the operators are not necessary normal.   

\subsection{Equivalent formulation of the dynamical sampling problem}
Using the fact that $\la Af,g\ra=\la f,A^*g\ra$, we get the following equivalent formulation of the dynamical sampling problem described in Section \ref {DSP}.

 \begin{proposition}
\begin{enumerate}
\item Any $f\in \HH$ can be recovered from  $ \left\{\la A^nf,g \ra \right\}_{g\in \G, 0\le n < L(g)} $ if and only if the system $ \{(A^*)^ng\}_{g\in \G, 0\le n < L(g)} $ is complete in   $ \HH $.
\item Any $f\in \HH$ can be   recovered from  $ \left\{\la A^nf,g \ra \right\}_{g\in \G, 0\le n < L(g)} $ in a stable way if and only if the system $ \{(A^*)^ng\}_{g\in \G, 0\le n < L(g)} $ is a frame in $ \HH $.
\end{enumerate}
 \end{proposition}
 Because of this equivalence, we drop the $*$ and we investigate systems of iterations of the form $ \{A^ng\}_{g\in \G, 0\le n < L(g)} $, where $A$ is a bounded operator on the Hilbert space $\HH$, $\G$ is a subset of  $\HH$, and $L$ is a function from $G$ to the extended set  of integers $\N^*=\N\cup \{\infty\}$.  The goal is then to find conditions on $A$, $\G$ and $L$ so that $ \{A^ng\}_{g\in \G, 0\le n < L(g)} $ is complete, Bessel, a basis, Riesz basis, frame, etc. In the remainder of this chapter, we only study the case where $L(g)=\infty$, for each $g \in \G$.

\subsection{Normal operators}
\label {NO}
Theorem \ref {TAD1} as well as most of the results in \cite {ACMT14} have been generalized to the case of normal operators in general Hilbert spaces \cite{ACCMP15}, and new results have been obtained.  The work relied on the spectral theorem of normal operators in Hilbert spaces (see e.g., \cite{Con96} Ch. IX, theorem 10.16, and \cite{conway1}). Since we will use this theorem again in this work, we state a version of this landmark theorem and give an example to clarify its meaning. In essence, the spectral theorem of normal operators is a way of diagonalizing any normal operator in a separable complex Hilbert space. Using an appropriate unitary transformation $U$,  a normal operator $A$ is equivalent to a multiplication operator $UAU^{-1}\tilde f=N_{\mu}\tilde f=z\tilde f$ where $\tilde f$  is a vector valued function on $\CC$, and $N_{\mu}\tilde f(z)=z\tilde f(z)$ for every $z\in \CC$. Specifically,

	\begin{theorem} [Spectral theorem with multiplicity]\label{decomp} For any normal operator $A$ on $\HH$ there are mutually singular compactly supported Borel measures  $\mu_j,\;1\leq j\leq\infty$ on $\CC$, such that $A$ is  equivalent to the operator
		$$ N^{(\infty)}_{\mu_\infty}\oplus N_{\mu_1}\oplus N^{(2)}_{\mu_2}\oplus\cdots$$
		i.e. there exists a unitary operator $U$  $$U:\HH\rightarrow\widetilde {\HH}=(L^2(\mu_\infty))^{(\infty)}\oplus L^2(\mu_1)\oplus (L^2(\mu_2))^{(2)}\oplus\cdots$$ such that	
		\begin{equation} \label{repres}
		UAU^{-1}=N_\mu=N^{(\infty)}_{\mu_\infty}\oplus N_{\mu_1}\oplus N^{(2)}_{\mu_2}\oplus\cdots.
		\end{equation}	
		Moreover, if $M$ is another normal operator with corresponding measures $\nu_\infty,\nu_1,\nu_2,\dots$ then $M$ is  equivalent to $A$ if and only if  $[\nu_j]=[\mu_j],\;j=1,\dots,\infty$ (are mutually absolutely continuous).
	\end{theorem}
Since the measures $\{\mu_j: \, j \in \N^*\}$ are mutually singular, we can define the measure $\mu=\sum_j\mu_j$ on $\CC$. A function  $\wi g\in \widetilde \HH$ is a vector of the form $(\wi{g}_j)_{j\in \N^*}$, where 
$\wi{g}_j$ is the restriction of $\wi g$ to $(L^2(\mu_j))^{(j)}$. 

\begin {remark}\label {ptwise}
Note that for every $1\le j< \infty$, $\wi g_j(z)$ is a finite dimensional vector in $l^2\{1,\dots,j\}$ and for $j=\infty$, $\wi g_\infty (z)$ is a vector in $l^2(\N)$. In order to simplify notation, we define  $\Om_j$ to be the set $ \{1,...,j\}$  and $\Om_\infty$ to be the set $\N$. Note that $l^2(\Om_j) \cong \CC^j$, for $j\in \N$, and $l^2(\Omega_\infty) = l^2(\N).$ For $j=0$ we define $l^2(\Omega_0)$ to be the trivial space $\{0\}$.

\end {remark}

An example to clarify the use of the Theorem above is given below.
\begin {example}
Let $A$ be the $8 \times 8$  diagonal matrix
 \[ A= \left( \begin{array}{ccc}
	\lambda_1 I_2& 0 &0 \\
	0 &  \lambda_2 I_3 &0\\
	0 & 0  & \lambda_3 I_3\\ 
	\end{array} \right)\]
where $\lambda_i\neq\lambda_j$ if $i\ne j$ and $I_j$ denotes the $j\times j$ identity matrix. For this case, the theorem above gives:  $\widetilde {\HH}= (L^2(\mu_2))^{(2)}\oplus (L^2(\mu_3))^{(3)}$, $\mu_2=\delta_{\lambda_1}$, $\mu_3=\delta_{\lambda_2}+\delta_{\lambda_3}$, where $\delta_{x}$ is the Dirac measure at $x$. If $g=(g_1,\dots,g_{8})^T$, then  $Ug=\tilde g=\big(\tilde g_j\big)$. In particular, $\wi g_3(\lambda_2)=\left( \begin{array}{c}
g_3\\
g_4\\
g_5
\end{array} \right)$, $\wi g _3(\lambda_3)=\left( \begin{array}{c}
g_6\\
g_7\\
g_8
\end{array} \right)$ and  $\wi g_3(z)=\vec 0$ for $z\ne \lambda_2, \lambda_3$ (in fact for $z\ne \lambda_2, \lambda_3$, $\wi g_3(z)$ can take any value since the measure $\mu_3$ is concentrated on  $\{\lambda_2, \lambda_3\} \subset \CC$). We have 
\begin {eqnarray*}\langle Uf,Ug\rangle&=&\int_{\CC} \langle \tilde f(z),\tilde g(z)\rangle d\mu(z)\\
&=&\int_{\CC} \langle \tilde f_2(z),\tilde g_2(z)\rangle d\mu_2(z)+\int_{\CC} \langle \tilde f_3(z),\tilde g_3(z)\rangle d\mu_3(z)\\
&=&\langle \tilde f_2(\lambda_1),\tilde g_2(\lambda_1)\rangle+ \langle \tilde f_3(\lambda_2),\tilde g_3(\lambda_2)\rangle+\langle \tilde f_3(\lambda_3),\tilde g_3(\lambda_3)\rangle\\
&=&\sum\limits_{j=1}^8f_j\overline{g}_j=\langle f, g \rangle.
\end{eqnarray*}
\end {example}	
	
	Since the measures $\mu_j$ in Theorem \ref {TAD1} are mutually singular, there are mutually disjoint Borel sets $\{\mathcal{E}_j\}$ such that $\mu_j$ is concentrated on $\mathcal{E}_j$ for every $1\leq j\leq \infty$. 
	
	The function $n:\CC\rightarrow \{1,2,\dots,\infty\}$ given by
	\[ n(z)=\begin{cases} 
	j, & z\in \mathcal{E}_j \\
	0, & otherwise 
	\end{cases}
	\]
	is called multiplicity function of the operator $A$.
	Thus every normal  operator is uniquely determined, up to a unitary equivalence,  by a pair $(n,[\mu])$ where $[\mu]$ is the class of  measures mutually singular with the compactly supported Borel measure $\mu$ and  $n:\CC\rightarrow \{1,2,\dots,\infty\}$ is a $\mu$ measurable function. 
 
\subsection {Action of normal operators  via infinite iterations}	
\label {ANOII}
	
In this section we present results from \cite {ACCMP15} about a system of infinite iterative action $\left\{A^ng_i\right\}_{ i\in I,\;n\ge0}$	 of a given normal operator $A\in  B(\HH)$ on a set of vectors $\G\subset \HH$.  Some of the results assume that $A$ is reductive, i.e., every invariant subspace $V$ for $A$ is also invariant for $A^*$.

	\begin{theorem} \label{mainthm} Let $A$ be a  normal operator on a Hilbert space $\HH$,  and let $\G$ be a countable set of vectors in $\HH$ such that  $\left\{A^ng\right\}_{ g \in \G,\;n\ge 0}$ is complete in $\HH$. Let  $\mu_\infty,\mu_1,\mu_2,\dots$ be the measures in  the  representation \eqref{repres}  of the operator $A$. Then for every  $1\leq j\leq\infty$ and $\mu_j$-a.e. $z$, the system of vectors  $\{\wi g_j(z)\}_{g\in \G}$ is complete in   $l^2\{\Om_j\}$.

 If in addition to being normal,  $A$ is also reductive,  then $\left\{A^ng\right\}_{ g\in \G,\;n\ge 0}$		being complete in $\HH$ is equivalent to $\{\wi g_j(z)\}_{g\in \G}$ being complete in   $l^2\{\Om_j\}$  $\mu_j$-a.e. $z$  for every  $1\leq j\leq\infty$.	
\end{theorem}

Although, the system of iteration $\left\{A^ng\right\}_{ g\in \G,\;n\ge 0}$ is complete, it is shown in \cite {ACCMP15} that it cannot be a basis for $\HH$. The obstruction is that $\left\{A^ng\right\}_{ g\in \G,\;n\ge 0}$ cannot be  minimal and complete at the same time. 
\begin {theorem} \label {NRB}
If  $A$ is a  normal operator on $\HH$ then, for any set of vectors  $\G \subset \HH$,    the system of iterates $\left\{A^ng\right\}_{g\in\G, n\ge 0}$ is not a basis for $ \HH $.
\end {theorem}
\begin {remark}
The normality assumption on $A$ is essential. For example, if $S$ is the right-shift operator on $\HH=l^2(\N)$, then $\{S^ne_0\}_{n\ge 0}$ is an orthonormal basis for $\HH=l^2(\N)$. In fact, it can be shown that, in a Hilbert space, a system  of vectors  $\HH$ $\{T^ng\}_{n\ge 0}$ generated by  $T\in B(\HH)$ and $g \in \HH$ is a Riesz basis if and only if is unitarily equivalent to the right-shift operator $S$ in $l^2(\N)$ \cite {IK16}.

\end{remark}
	
The fact that, for a normal operator $A$, $\left\{A^ng\right\}_{ g\in \G,\;n\ge 0}$ cannot be basis is that when it is complete, it must be redundant (since it is not minimal). But it is possible for such a sequence to be a frame. For example, the following theorem characterizes frames generated by the iterative action  of diagonalizable normal operators acting on a single vector $b$ \cite {ACMT14}.

\begin{theorem} \label {OnePointFrame}
Let  $\Lambda=\sum_j\lambda_jP_j$, acting on $l^2(\N)$, be such that $P_j$ have rank $1$ for all $j\in \N$, and let $b := \{b(k)\}_{k \in \N} \in l^2(\N)$. 
Then $\{\Lambda^lb: l=0,1,\dots\}$ is a frame if and only if
\begin{enumerate}
\item[i)] $|\lambda_k| < 1$  for all $k.$  
\item[ii)]  $|\lambda_k| \to 1$.
\item[iii)] $\{\lambda_k\}$ satisfies Carleson's condition 
\begin{equation}
\label{carleson-cond}
\inf_{n} \prod_{k\neq n} \frac{|\lambda_n-\lambda_k|}{|1-\bar{\lambda}_n\lambda_k|}\geq \delta.
\end{equation}
for some $\delta>0$.
\item[iv)] $b(k)=m_k\sqrt{1-|\lambda_k|^2}$ for some sequence $\{m_k\}$ satisfying $0<C_1\le |m_k| \le C_2< \infty$.
\end{enumerate}
\end{theorem}
In the previous theorem, the spectrum lies inside the unit disk $D_1$. Moreover, the spectrum concentrates near its boundary $S_1$. These facts can be generalized for normal operators \cite {ACCMP15}.
\begin{theorem} \label{framecond} Let $A$ be a normal operator on an infinite dimensional  Hilbert space $\HH$ and $ \G $ a system of vectors in $ \HH $. 	
\begin{enumerate}
		\item If  $\{A^n g\}_{g\in \G, \;n\geq 0}$ is complete in $\HH$  and for every $g\in \G$ the system  $\{A^n g\}_{n\geq 0}$ is Bessel in $\HH$, then	$\mu\left(D_1^c\right)=0$ and $\mu|_{S_1}$ is absolutely continuous with respect to  arc-length measure (Lebesgue measure) on $S_1$.
		\item \label{condb}	If  $|\G|<\infty$ and  $\{A^n g\}_{g\in \G, \;n\geq 0}$ satisfies the lower frame bound then, for every $0<\epsilon<1$, $\mu\left(D_{1-\epsilon}^c\right)>0$, where $D_{1-\epsilon}$ is the closed disc of radius ${1-\epsilon}$.
	\end{enumerate}
\end{theorem} 
It can be proved that if $\mu\left(D_1^c\right)=0$ and $\mu|_{S_1}$ is absolutely continuous with respect to  arc-length measure on $S_1$ then there exists a set $G\subset\HH$ such that $\{A^n g\}_{g\in \G, \;n\geq 0}$ is complete and Bessel system in $\HH$. Other developments on this theme can be found in \cite {PF16}.

\begin{corollary}
		If  for a normal operator $A\in B(\HH)$ in an infinite dimensional space $\HH$ the system of vectors  $\{A^n g\}_{g\in \G, \;n\geq 0}$ with $|\G|<\infty$ is a frame, then $A$ is unitarily equivalent to an operator $\Lambda=\sum_j\lambda_jP_j$ where $ P_j $ are projections such that $\dim P_j\leq |\G|$. In particular, if $|\G|=1$, then $\lambda_j$ satisfy conditions $i),ii)$ in  Theorem \ref{OnePointFrame}.
	\end{corollary}
	\begin {proof} Define the subspace $\widetilde V_\rho$ of $\widetilde \HH$ to be $\widetilde V_\rho=\{\tilde f: \supp \tilde f \subseteq D_\rho\}$. The restriction  of $UAU^*$  to $\widetilde V_\rho$ is normal with its spectrum equal to the part of the spectrum of $A$ inside $D_\rho$.   If we iterate the $z$-multiplication operator on the projections $\G_\rho=P_{\widetilde V_\rho}\G$ of the vectors in $\G$ we  get a frame for $\widetilde V_\rho$ hence, from part 2), $\widetilde V_\rho$ must be finite dimensional. That implies the spectrum of $A$ is finite inside every $D_\rho$ with $\rho<1$. We also know from Part (1) of Theorem \ref {framecond} that $\mu(D_1^c)=0$. Furthermore, from Corollary 
	\ref {nounitfrzame} below, $\mu(S_1)=0$.  Thus,  $UAU^*$ has the form $\Lambda=\sum_j\lambda_jP_j$. The fact that $\dim P_j\leq |\G|$ follows from Theorem \ref {TAD1}. The rest follows from Theorem \ref{OnePointFrame}. \qed
	\end {proof}

\subsection{New results for general bounded operators}	
\label {NRGBO}

This section is devoted to the study of the iterative action of general bounded operators in $B(\HH)$. 
	\begin{theorem}\label{noframeany}
			 If for an operator  $A\in B(\HH)$  there exists a set of vectors $\G $  in $ \HH $ such that $\{A^n g\}_{g\in \G, \;n\geq 0}$ is  a frame in $\HH$ then   for every $f\in \HH$, $(A^*)^nf\to 0$ as $n\to \infty$.

	\end{theorem}
	\begin{proof} Suppose, for some $\{g\}_{g\in G} $, $\{A^n g\}_{g\in \G, \;n\geq 0}$ is a frame with frame bounds $B_1$ and $ B_2 $.
		Let $f \in \HH$.  Then  for any $m \in \Z $ we  have
		
		\begin {eqnarray}
		\label {UP}
		\sum_{g\in \G}\sum_{n= 0}^{\infty}|\la(A^*)^mf,A^ng\ra|^2&=&\sum_{g\in \G}\sum_{n= 0}^{\infty}|\la f,A^{n+m}g\ra|^2\\
		&=&\sum_{g\in \G}\sum_{n= m}^{\infty}|\la f,A^{n}g\ra|^2. \nonumber
		\end {eqnarray}
		Since $\sum_{g\in \G}\sum_{n= 0}^{\infty}|\la f,A^{n}g\ra|^2\le B_2\|f\|^2$, we conclude that $\sum_{n= m}^{\infty}\sum_{g\in \G}|\la f,A^{n}g\ra|^2\to 0$ as $m\to \infty$. Thus, from \eqref{UP}, we get that 
		$\sum_{g\in \G}\sum_{n= 0}^{\infty}|\la(A^*)^mf,A^ng\ra|^2\to 0$ as $m\to \infty$. Using the lower frame inequality, we get

				$$B_1\|(A^*)^mf\|\leq \sum_{g\in \G}\sum_{n=0}^{\infty}|\la(A^*)^mf,A^ng\ra|^2.$$
		Since the right side of the inequality tends to zero as $m$ tends to infinity we get that  $(A^*)^mf\to 0$ as $m\to \infty$.\qed
	\end{proof}

	\begin{corollary}
	\label {nounitfrzame}
		For any unitary operator  $A:\HH\to \HH$ and any set of vectors $\G\subset \HH $, $\{A^n g\}_{g\in \G, \;n\geq 0}$ is not a frame in $\HH$.
	\end{corollary}
	
	If for every $f\in \HH$, $(A^*)^nf\to 0$ as $n\to \infty$, then we can get the following  existence theorem of frames for $\HH$ from  iterations.
	
	\begin{theorem}
If $ A $ is a contraction (i.e., $ \|A\|\leq 1 $), and  for every $f\in \HH$, $(A^*)^nf\to 0$ as $n\to \infty$,  then we can choose $\G\subseteq \HH $  such that $\{A^n g\}_{g\in \G, \;n\geq 0}$ is a tight frame.
	\end{theorem}
	
		\begin{remark}
	The system we find in this case  is not very useful since the initial system $\G$ is 'too large' (it is complete in $ \HH $ in some cases).  Moreover, the  condition $\|A\|\leq 1$  is not necessary  for the existence of a frame with iterations. For example, we can take nilpotent operators with large operator norm for which there are frames with iterations.  
	\end{remark}

	\begin{proof}
 Suppose for any $ f\in\HH $, $(A^*)^nf\to 0$ as $n\to \infty$ and $ \|A\|\leq 1 $. Let $D=(I-AA^*)^{\frac{1}{2}}$ and $\V=cl (D\HH)$. Let $\{h\}_{h\in \mathcal I}$ be an orthonormal basis for $ \V $. Then
 
 \begin {eqnarray*} \sum_{n=0}^{m}\sum_{h\in \mathcal I}|<f,A^nDh>|^2&=&\sum_{n=0}^{m}\sum_{h\in \mathcal I}|<D(A^*)^nf,h>|^2\\
 &=&\sum_{n=0}^{m}\|D(A^*)^nf\|^2\\
 &=&\sum_{n=0}^{m}<D^2(A^*)^nf,(A^*)^nf>\\
 &=&\sum_{n=0}^{m}<(I-AA^*)(A^*)^nf,(A^*)^nf>\\
 &=&\|f\|^2-\|(A^*)^{m+1}f\|.
 \end {eqnarray*}
 Taking limits as $m\to \infty$ and using the fact that $(A^*)^mf\to 0$ we get from the  identity above that 
 
 $$\sum_{n=0}^{\infty}\sum_{h\in \mathcal I}|<f,A^nDh>|^2=\|f\|^2.$$
 Therefore  the system of vectors $\G=\{g=Dh: h \in \mathcal I\}$ is a tight frame for $\HH$.
 \qed
	\end{proof}
	
	\begin{theorem}
		If $\dim \HH=\infty$, $|\G|<\infty$, and $\{A^n g\}_{g \in \G, \;n\geq 0}$  satisfies the lower frame bound, then $\|A\|\geq 1$.
	\end{theorem}
	\begin{proof}
		Suppose $\|A\|< 1$. Since $\{g\}_{g \in \G} $ is finite and $\dim(\HH)=\infty$, for any fixed $N$ there exists a vector $f\in \HH$ with $\|f\|=1$ such that $<A^n g,f>=0$, for every $g\in \G$ and  $0\leq n\leq N$. Then 
		$$\sum_{g \in \G}\sum_{n\geq 0}|<A^n g,f>|^2=\sum_{g \in \G}\sum_{n= N}^{\infty}|<A^n g,f>|^2\leq\sum_{g \in \G}\|g\| \sum_{n= N}^{\infty}\|A\|^{2n}\to 0$$
		as $N\to \infty$ hence the lower frame bound cannot hold.\qed
	\end{proof}
	
	\begin{corollary}
		Let $\{A^n g\}_{g \in \G, \;n\geq 0}$ with $|\G|<\infty$ satisfies the lower frame bound. Then for any  coinvariant subspace $ \V \subset \HH$ of $A$ with $\|P_\V AP_\V\|<1$ we have that $\dim(\V)<\infty$.
	\end{corollary}
	\begin{proof} $ \V $ is coinvariant for $ A $ is equivalent to
		$$P_\V A=P_\V AP_\V.$$
		It follows that $P_\V A^n=P_\V A^nP_\V$. Hence, if $\{A^n g\}_{g\in \G, \;n\geq 0}$ satisfies the lower frame inequality in $\HH$, then $\{(P_\V AP_\V)^n g\}_{g\in \G, \;n\geq 0}$ also satisfies the lower frame inequality for $ \V $ and hence from the previous theorem if $\dim(\V)=\infty$, then $\|P_\V AP_\V\|\geq1$.\qed
	\end{proof}

\section{Related work and concluding remarks}

There are several  features that are particular to the present work: In the system of  iterations $ \{(A^*)^ng\}_{g\in \G, 0\le n < L(g)} $ that we considered in this chapter, we let $L(g)=\infty$ for all $g \in \G$. This setting implies strong constraints on the spectrum of $A$ when we further require that the system is a Bessel system, a frame, etc. Since in finite dimensional spaces every finite spanning set is a frame, and since for fixed $g$, if $K> \dim (\HH)$, then the set $ \{(A^*)^ng\}_{g\in \G, 0\le n \le K} $  is always linearly dependent, it does not make sense to let $L(g)>  \dim (\HH)+1$. In fact, the finite dimensional problem has first been studied \cite {ADK13} in which $L(g)$ is a constant for all $g\in\G$ and is as small as possible in some sense. 

\begin{acknowledgement}
This work has been partially supported by NSF/DMS grant  1322099. Akram Aldroubi would like to thank Charlotte Avant and Barbara Corley for their attendance to the comfort and entertainment during the preparation of this manuscript.
\end{acknowledgement} 

\bibliographystyle{amsplain}
\bibliography{refers}
\end{document}